\documentclass[12pt]{amsart}
\usepackage{amssymb, amsmath, pinlabel, graphicx, amscd, multirow}
\usepackage{hyperref}

\newtheorem{thm}{Theorem}[section]

\newtheorem{defi}[thm]{Definition}

\newtheorem{lma}[thm]{Lemma}

\newtheorem{rem}[thm]{Remark}

\numberwithin{equation}{section}

\newcommand{\R}{{\mathbb R}}
\newcommand{\de}{d}

\newcommand{\A}{\operatorname{area}}
\newcommand{\id}{\operatorname{id}}

\begin{document}

\title{Area preserving isotopies of self transverse immersions of $S^1$ in $\R^2$. }
\author{Cecilia Karlsson}
\address{Department of mathematics, 
Uppsala University, 
Box 480, 
751 06 Uppsala, 
Sweden}
\email{ceka{\@@}math.uu.se}

\begin{abstract}
Let $C$ and $C'$ be two smooth self transverse immersions of $S^1$
into $\R^2$. Both $C$ and $C'$ subdivide the plane into a number of disks and one
unbounded component. An isotopy of the plane which takes $C$ to $C'$ induces
a 1-1 correspondence between the disks of $C$ and $C'$. An obvious necessary
condition for there to exist an area-preserving isotopy of the plane
taking $C$ to $C'$ is that there exists an isotopy for which the area of every
disk of $C$ equals that of the corresponding disk of $C'$. In this paper
we show that this is also a sufficient condition.
\end{abstract}

\subjclass[2000]{57R17 (53C44, 53D42)}

\maketitle 

\section{Introduction}\label{sec:int}
Let $C$ be a smooth self transverse immersion of $S^1$ into the plane $\R^2$ (by Sard's theorem any immersion is self transverse after arbitrarily small perturbation).  Then $C$  subdivides the plane into a number of bounded connected components and one unbounded component. The bounded components are topological disks and we call them the  {\em disks of $C$}. Let $C'$ be another self transverse immersion of $S^1$ into $\R^{2}$ such that there exists an isotopy of the plane taking $C$ to $C'$. Then the isotopy induces a 1-1 correspondence between the disks of $C$ and the disks of $C'$. 

In this paper we study the existence of area-preserving isotopies of the plane taking $C$ to $C'$, where,  if $\de x\wedge \de y$ denotes the standard area form on $\R^{2}$, we say that an isotopy $\phi_\tau\colon\R^{2}\to\R^{2}$, $0\le \tau\le 1$ is {\em area-preserving} if 
\begin{equation*}
 \phi_\tau^{\ast}(\de x\wedge \de y)=\de x\wedge \de y
\end{equation*}
for every $\tau\in[0,1]$. Since $\phi_\tau$ area-preserving implies that  $\A(\phi_\tau(U))=\A(U)$ for any measurable $U \subset \R^2$, an obvious necessary condition for the existence of an area-preserving isotopy $\phi_\tau$ taking $C$ to $C'$ is that the area of any disk  $D$  of $C$ satisfies
\begin{equation}\label{e:diskequal}
\A(D)=\A(D')
\end{equation}
where $D'$ is the disk of $C'$ which corresponds to $D$ under $\phi_\tau$. We call an isotopy which satisfies \eqref{e:diskequal} \emph{disk-area-preserving}. The main result of the paper shows that this is also a sufficient condition. More precisely, we have the following result.  
\begin{thm}\label{thm1}
Let $C$ and $C'$ be two self transverse immersions of $S^1$ into $\R^2$ and assume that there is a disk-area-preserving isotopy $\psi_\tau$, $0\le\tau\le 1$, of $\R^{2}$ taking $C$ to $C'$ (i.e., $\psi_0=\id$, $\psi_1(C)=C'$, and for every disk $D$ of $C$, $\A(\psi_1(D))=\A(D)$). Then there exists an area-preserving isotopy $\phi_\tau$, $0\le \tau\le 1$, of $\R^{2}$ with $\phi_0=\id$ and  $\phi_1(C)=C'$.
\end{thm}  

Theorem \ref{thm1} is proved in Section \ref{sec:proof}. Problems related to the existence of a topological isotopy (without area condition) taking $C$ to $C'$ were studied by many authors, see e.g.~\cite{ Carter, Vassiliev, Merkov}.

From the point of view of symplectic geometry $C$ is an immersed Lagrangian submanifold, and area-preserving isotopies are Hamiltonian isotopies. For related questions in higher dimensions see e.g.~\cite{Hind2, Hind1, Georgios}.

In short outline, our proof of Theorem \ref{thm1} is as follows. First, we construct an isotopy $\chi_\tau$ which takes $C$ to $C'$ and such that for every disk $D$ of $C$ we have $\A(\chi_\tau(D))=\A(D)$ for all $\tau$. We call such an isotopy \emph{semi-area-preserving with respect to $C$}. The semi-area-preserving isotopy is constructed from the disk-area-preserving isotopy $\psi_\tau$ by first composing it with a time dependent scaling so that the resulting isotopy $\gamma_\tau$ shrinks the area of each disk of $C$ for all times. The  isotopy $\gamma_\tau$ is then modified: we introduce a time-dependent area form $\omega_\tau$ such that the area of every disk of $C$ is constant under $\gamma_\tau$ with respect to $\omega_\tau$ and then we use Moser's trick to find an isotopy $\phi_\tau$ such that $\phi_\tau ^* \de x \wedge \de y = \omega_\tau$, and hence the isotopy $\phi_\tau \circ \gamma_\tau $ is semi-area-preserving, see Section \ref{sec:sapi}. Second, we subdivide the semi-area-preserving isotopy into small time steps and use a cohomological argument to show the existence of an area-preserving isotopy, see Section  \ref{sec:delta}.  
 
For simpler notation below, we assume that all maps are smooth and that all immersions are self transverse.

\subsection*{Acknowledgements}
I would like to thank Tobias Ekholm for helpful discussions and for supervising the master thesis on which this paper is based. I would also like to thank Georgios Dimitroglou Rizell for the central ideas of the proof of Lemma \ref{moser2} and Lemma \ref{sapi}.

\section{Background}\label{sec:back}
In this section we introduce notation and discuss background material on Hamiltonian vector fields on surfaces.

Let $M$ be a surface and let $v:M\rightarrow TM$ a vector field with compact support. We write $\Phi_v^{t} \colon M\to M$ for the time $t$ flow of $v$.

Let $\omega$ be a symplectic form on $M$ and write $I:T^*M\rightarrow TM$ for the isomorphism defined through the equation
\begin{equation*}
 \alpha(\eta)=\omega(\eta,I(\alpha)) \text{ for all } \alpha\in T^{\ast}M,\;\eta\in T_xM. 
\end{equation*} 
Let $H:M\rightarrow \R$ be a smooth function with compact support. The vector field $X_H=I(dH)$ is the \emph{Hamiltonian vector field} of $H$ and its flow is area-preserving.

Let $C$ be an immersion of $S^1$ into the plane and let $\varphi:S^1\rightarrow \R^{2}$ be a parametrization of $C$. Write $e(s)$ for the unit vector field along $C$ such that $ \left(\frac{\de\varphi}{\de s}(s),e(s)\right)$ is a positively oriented basis of $\R^{2}$ for all $s\in S^1$. Then for all sufficiently small $\epsilon>0$ the map $\Phi : S^1\times (-\epsilon,\epsilon) \rightarrow \R^2$, 
\begin{equation}\label{e:nbhdmap}
\Phi(s,t)=\varphi(s)+te(s)
\end{equation}  
parametrizes a neighborhood $C^\epsilon$ of $C$. Notice that if $C$ has double points then this parametrization is not one-to-one. 

Let $\de x \wedge \de y$ be the standard symplectic form on $\R^{2}$ and consider coordinates $(s,t)$ on $S^{1}\times\R=(\R/2\pi\mathbb{Z})\times\R$ with the corresponding symplectic form $\de s \wedge \de t$. The following lemma is a special case of Moser's lemma, see e.g.~\cite{moser} for a proof. 

\begin{lma}\label{moser}
Let $C$ be an immersion of $S^{1}$ in $\R^{2}$ and let $\Phi$ be as in \eqref{e:nbhdmap}. Then there exists $\delta>0$ and a diffeomorphism $\vartheta:S^1\times \R \rightarrow S^1\times \R$ with $\vartheta(s,0)=(s,0)$ such that 
\begin{equation*} 
(\Phi \circ \vartheta)^* \de x \wedge \de y = \de s \wedge \de t,
\end{equation*}
for all $|t|<\delta$.
\end{lma}

Below we will often combine Lemma \ref{moser} with a Hamiltonian isotopy
of $S^1\times \R$.
In the following lemma we use this argument to construct area-preserving isotopies between nearby curves $C$ and $C'$ which agree near double points. 
We will use the following terminology: For $C \subset \R^2$ an immersed circle, we call an arc $A \subset C$ a \emph{maximal smooth arc} of $C$  if  $A \cap \{x_i\}_{i=1}^{n} = \{x_i,x_j\}= \partial A$, where  $\{x_i\}_{i=1}^{n}\subset C$ are the double points of $C$.
 
\begin{lma}\label{lma5}
Let $C$ be an immersion of $S^1$ into $\R^2$ and let $\xi:S^1 \times(-\epsilon, \epsilon)\rightarrow \R^2$ be an area-preserving parametrization of a neighborhood $C^\epsilon$ of $C$ as in Lemma \ref{moser}. Assume that $C'$ is an immersion of $S^{1}$ into $\R^{2}$ which coincides with $C$ in a neighborhood $U_x$ of every double point $x$ of $C$ and such that there is a function $g: S^1 \rightarrow (-\epsilon, \epsilon)$  with $C'=\xi(\Gamma)$, where  $\Gamma$ is the graph of $g$. If there exists a disk-area-preserving isotopy taking $C$ to $C'$ then there exists an area-preserving isotopy of the plane taking $C'$ to $C$.
\end{lma}
\begin{proof}
Shrink $C^\epsilon$ so that we still have $C \cup C' \subset C^\epsilon$, but so that the parametrization is 1-1 outside $\bigcup U_x$, i.e.\ so that $\overline{C^\epsilon- \bigcup U_x}$ consists of a number of simply connected components $V_A$ where each component corresponds to a maximal smooth arc $A$ of $C$.
Let $W$ be an open neighborhood  of $C \cup C'$ so that $\overline{W} \subset C^\epsilon$ and so that $V_A \cap W$ and $U_x \cap W$ are simply connected for all $V_A, U_x$.  Let $G : S^1 \rightarrow \R$ be defined by $G(s) = \int_0 ^s {g(s') \de s'}$, and let $\tilde{G}: \R^2 \rightarrow \R$ be a function satisfying 
\begin{equation*}
\tilde{G}(x) =
\begin{cases}
G((\xi^{-1})^1(x)) & \text{for }x \in W \cap V_A \\
G((\xi^{-1})^1(x')) & \text{for }x \in W \cap U_{x'} \\
0		& \text{for } x \notin C^\epsilon				
\end{cases}
\end{equation*}   
where $\xi^{-1}= ((\xi^{-1})^1, (\xi^{-1})^2)$.
Then $\tilde{G}$ is a well-defined function: Suppose that $\xi(s_1,t_1) = \xi(s_2,t_2)$ for $s_1 \neq s_2$. Then $\xi(s_1,t_1) \subset U_x$ for some $x$, and since $\tilde{G}$ is constant in $U_x\cap W$ we can assume that $\xi(s_1,t_1)=x$.  But clearly $\xi((s_1,s_2) \times\{0\})$ is a 1-chain, so it bounds a number of disks of $C$. Since every disk of $C'$ has the same area as the corresponding disk of $C$ we thus have $ \int_{s_1}^{s_2} {g(s) \de s} = 0 $, so $G(s_1) = G(s_2)$.

The Hamiltonian vector field of $\tilde{G}$ in the parametrization of $C^ \epsilon$ is $X_{\tilde G}=-g(s)\,\frac{\partial}{\partial t}$ for $(s,t) \in W \cap V_A$ and $X_{\tilde G}=0$ in $U_x \cap W$. Hence its time $1$-flow takes $(s,g(s))$ to $(s,0)$ for all $s$ and we get an area-preserving isotopy of the plane taking $C'$ to $C$.
\end{proof}

\section{Construction of semi-area-preserving isotopies}\label{sec:sapi}
In this section we construct a semi-area-preserving isotopy from a disk-area-preserving isotopy.

Let $C$ and $C'$ be two immersions of $S^1$ into $\R^2$ such that there exists a disk-area-preserving isotopy $\phi_t$ taking $C$ to $C'$. 
Without loss of generality we can assume that $\phi_t$ has support in some $B_r$, where $B_r$ denotes the open disk of radius $r$ centered at $0$. Let $\gamma_t:\R^2 \rightarrow \R^2$, $t\in[0,1]$, $\gamma_0=\id$, be an isotopy of the plane with support in  $B_{r+1}$, acting as follows. First let $\gamma_t$ shrink $B_r$ to some $B_{\epsilon r}$ radially, where $\epsilon $ is small and depends on the area of the disks of $C$. Next we let $\gamma_t$ take the shrunken curve $C$ to the shrunken curve $C'$ by using $\epsilon\phi_t(\epsilon x)$, and then finally we let $\gamma_t$ enlarge $B_{\epsilon r}$ to $B_r$ again, so that we get $\gamma_1(C)=C'$. By choosing $\epsilon$ small enough we thus get an isotopy $\gamma_t$ of the plane taking $C$ to $C'$ such that $\A(\gamma_t(D)) < \A(D)$ for every disk $D$ of $C$, and for all $t \in (0,1)$. 

Next we use Moser's trick to find an  isotopy $\psi_t:\R^2 \rightarrow \R^2,t \in [0,1], \psi_0=\id$,  such that $\chi_t= \psi_t \circ \gamma_t$ is semi-area-preserving with respect to $C$. So if we then can take $\psi_1 \gamma_1(C)$ to $C'$ with a semi-area-preserving isotopy we get a semi-area-preserving isotopy taking $C$ completely to $C'$. We start with the following lemma.

\begin{lma}\label{moser2}
Let $\gamma_t, C,C'$ be as above. Then there is an isotopy $\psi_t:\R^2 \rightarrow \R^2 ,t \in [0,1], \psi_0=\id$,  such that $\int_{\psi_t \gamma_t (D)}{\de x \wedge \de y}=\int_{D}{\de x \wedge \de y}$ for every disk $D$ of $C$. Moreover, $\psi_t$ can be chosen so that $\psi_1^* \de x \wedge \de y = \de x \wedge \de y$. 
\end{lma}
\begin{proof}
Let $D_1,D_2,\dotsc,D_n$ be the disks of $C$. For each $D_i$ choose a point $\xi_i \in D_i$, and let $r_i(t):[0,1] \rightarrow (0, \infty)$ be such that $B_{r_i(t), \gamma_t(\xi_i)}\subset \gamma_t(D_i)$ for all $0 \le t \le1$, where $B_{\rho, p}$ is the open disk of radius $\rho$ centered at $p$. 

For each disk $D_i$ let $\sigma_t^i:[0,\infty) \rightarrow (0,\infty)$ be a smooth 1-parameter family of functions such that for each $t\in[0,1]$ we have, if $(\rho, \theta)$ are polar coordinates centered at $\gamma_t(\xi_i)=(x_i(t),y_i(t))$, that $\omega_t^i= \de(\frac{1}{2} \sigma_t^i(\rho^2) \de \theta)$ is nondegenerate and satisfies
\begin{equation}
\label{area}
\int_{\gamma_t(D_i)}{\omega_t^i}= \int_{D_i}{\de x \wedge \de y} - \frac{n-1}{n}\int_{\gamma_t(D_i)}{\de x \wedge \de y}.
\end{equation}
 Also choose $\sigma_t^i$ so that  
\begin{align}
\label{eq:part}
& \omega_t^i=\frac{1}{n}\de x \wedge \de y  \qquad \text{ in } B_r \setminus B_{r_i(t), \gamma_t(\xi_i)} \\ 
\label{end}
&\omega_0^i=\frac{1}{n}\de x \wedge \de y = \omega_1^i 
\end{align}
and so that $\sigma_t^i(s)=\frac{1}{n}s$ outside some $B_{r'}$, where  $r'>r$ is chosen big enough to be independent of $t$ and $D_i$. 

Such a $\sigma_t^i$ we can find due to the fact that we want $\omega_t^i$ to satisfy $\int_{\gamma_t(D_i)}{\omega_t^i} > \frac{1}{n}\int_{\gamma_t(D_i)}{\de x \wedge \de y}$. So even if the disk $B_{r_i(t), \gamma_t(\xi_i)}$ is small we can let $\frac{\de}{\de s} \sigma_t^i(s)$ be large in this disk to obtain \eqref{area}, which need not have been the case if the area of $\gamma_t(D_i)$ was greater than the area of $D_i$ for some $t$. We use the space between $B_r$ and $B_{r'}$ to decrease  $\frac{\de}{\de s} \sigma_t^i(s)>0$ so that we get $  \sigma_t^i(s)=\frac{1}{n}s$ outside $B_{r'}$.

Now let $\omega_t = \sum_{i=1}^n \omega_t^i$. Then 
\begin{align*}
\int_{\gamma_t(D_i)}{\omega_t} =& \int_{\gamma_t(D_i)}{\omega_t^i} +  \sum_{j=1, j \neq i}^n  \int_{\gamma_t(D_i)}{\omega_t^j} \\
=&\int_{D_i} \de x \wedge \de y - \frac{n-1}{n} \int_{\gamma_t(D_i)}\de x \wedge \de y +  \frac{n-1}{n} \int_{\gamma_t(D_i)}\de x \wedge \de y \\
=&\int_{D_i} \de x \wedge \de y. 
\end{align*}
So if we can find an isotopy $\psi_t$ satisfying $\omega_t=\psi_t^*\omega_0$ for all $t$ then $\psi_t \circ \gamma_t$ will be semi-area-preserving with respect to $C$. 

To do this we use Moser's trick. Namely, for each disk $D_i$ and for each $t$ let $\mu_t^i$ be the 1-form $\mu_t^i= \frac{\de}{\de t}(\frac{1}{2}\sigma_t^i(\rho^2) \de \theta)$, and let $v_t$ be the vector field defined by $\iota_{v_t}(\omega_t)+ \sum_{i=1}^n \mu_t^i =0$, where $\iota_{v_t}(\omega_t)$ is the 1-form satisfying $\iota_{v_t}(\omega_t)(\eta)=\omega_t(v_t, \eta)$ for all $\eta \in T_x\R^2$. Then we get that $v_t = \sum_{i=1}^n {\frac{\de x_i}{\de t}\frac{\partial}{\partial x}-  \frac{\de y_i}{\de t}\frac{\partial}{\partial y}}$ outside $B_{r'}$, since here we have that
\begin{equation*}
\sigma_t^i(\rho^2) \de \theta = \frac{1}{n}\rho^2 \de \theta = \frac{1}{n} ((x-x_i(t))\de y - (y-y_i(t))\de x)
\end{equation*}
so $\omega_t^i = \frac{1}{n}\de x \wedge \de y$ and $\mu_t^i= \frac{1}{n}( \frac{\de y_i}{\de t} \de x -  \frac{\de x_i}{\de t} \de y )$ here. Thus $v_t$ satisfies a Lipschitz condition with the same Lipschitz constant $L$ for all $x \in \R^2$ and for all $t\in [0,1]$, and hence we can find an isotopy $\chi_t:\R^2 \rightarrow \R^2, 0\le t \le 1$, such that $\chi_0 = \id$ and $\frac{\de \chi_t}{\de t} = v_t \chi_t$. Now we get 
\begin{align*}
\frac{\de}{\de t}(\chi_t^* \omega_t) = \chi_t^*(\de \iota_{v_t}(\omega_t)+ \sum{\de \mu_t^i})=0,
\end{align*}  
so $
\chi_t^*\omega_t = \chi_0^*\omega_0= \de x \wedge \de y \text{ for all  }t \in [0,1]$.
Letting $\psi_t$ be the inverse of $\chi_t$ for each $0 \le t \le 1$ we get that $\omega_t = \psi_t^* \de x \wedge \de y$ and hence that $\psi_t \circ \gamma_t$ is a semi-area-preserving isotopy with respect to $C$, and by \eqref{end} we have $\psi_1^* \de x \wedge \de y = \de x \wedge \de y$.
\end{proof}

Now by finding an area-preserving isotopy taking $\psi_1\gamma_1(C)$ to $\gamma_1(C)=C'$  we can prove the main lemma of this section.
\begin{lma}\label{sapi}
If $C$ and $C'$  are immersions of $S^1$ into $\R^2$ such that there exists a disk-area-preserving isotopy taking $C$ to $C'$, then there exists a semi-area-preserving isotopy with respect to $C$ taking $C$ to $C'$.
\end{lma}
\begin{proof}
Let $\gamma_t, \psi_t$ be constructed as above, and let
$F_t:\R^2 \rightarrow \R^2$, $0\le t \le 1$ be defined as
\begin{equation*}
F_t(x)= 
\begin{cases}
\frac{\psi_1(tx)}{t} & t \neq 0 \\
d \psi_1(0)x & t=0.
\end{cases}
\end{equation*}
Then $ d F_t(x) = d \psi_1(tx)$ for all $t$ and since $\psi_1^*\de x \wedge \de y =\de x \wedge \de y$ we get that $F_{1-t}$ is an area-preserving isotopy taking $\psi_1\gamma_1(C)$ to $d \psi_1(0)(\gamma_1(C))$. Moreover, since $\text{det}(d \psi_1(0))=1$ there is a one-parameter family of linear diffeomorphisms $A_t \in \text{SO}(2)$ such that $A_0 =  d \psi_1(0)$, $A_1 =\id$, and hence we can find an area-preserving isotopy of the plane taking $\psi_1\gamma_1(C)$ to $\gamma_1(C)=C'$. Since $\psi_t\circ \gamma_t$ is semi-area-preserving with respect to $C$ we thus get a semi-area-preserving isotopy of the plane taking $C$ to $C'$. 
\end{proof}

\section{Area preserving isotopies between nearby curves}\label{sec:delta}
In this section we show that if $C$ and $C'$ are two immersed circles in the plane such that there exists a disk-area-preserving isotopy taking $C$ to $C'$, if $C'$ lies sufficiently close to $C$ then there exists an area-preserving isotopy taking $C$ to $C'$. This implies that if we have two immersions $C$ and $C'$, not necessary close to each other, and a semi-area-preserving isotopy $\psi_\tau$ taking $C$ to $C'$, then we can find an area-preserving isotopy taking $C$ to $\psi_{\tau_0}(C)$ for $\tau_0$ sufficiently small. Thus, by compactness arguments, we can find an area-preserving isotopy taking $C$ completely to $C'$. 

We begin by finding a suitable parametrization of a neighborhood of $C$, and then we define what we mean by $C'$ being ``sufficiently close'' to $C$.

So given $C$,  let $\nu>0$ be so small that $\overline{B}_{\nu,x_1} \cap \overline{B}_{\nu,x_2} = \emptyset$ for any double points $x_1 \neq x_2$ of $C$. Let $\xi : S^ 1 \times (-\epsilon, \epsilon) \rightarrow \R^ 2$ be an area-preserving parametrization of a neighborhood $C^ \epsilon$ of $C$ as in Lemma \ref{moser}. Then at each double point $x$ of $C$ we get a double point of $\xi$, i.e.\ a subset $U_x \subset C^\epsilon$ where $C^\epsilon$ overlaps itself. Let $\epsilon$ be so small that $U_x$ is a disk contained in $B_{\nu,x}$ and so that $\overline{C\cap U_x} $ consists of two smooth arcs $ L_s, L_t$ intersecting at $x$. Suppose that $x= \xi(0,0)$ and that $L_s = \xi ([-s_1,s_1] \times \{0\})$. Since $L_s$ intersects $L_t$ transversely at $x$ there is a $t_1 > 0$ so that $L_t \cap ((-s_1,s_1) \times (-t_1,t_1))$ coincides with  the graph of a function $g: (-t_1, t_1) \rightarrow (-s_1,s_1)$ over the $t$-axis in the parametrization of $C^\epsilon$. Let $S = (-s_1,s_1) \times (-t_1,t_1)$ and let $\vartheta: S \rightarrow \R^2$ be defined by 
\begin{equation*}
\vartheta(s,t) = (\mu(s,t), \eta(s,t))= (s- g(t) ,t ). 
\end{equation*}
Then $\vartheta ^* \de \mu \wedge \de \eta = \de s \wedge \de t $, and $\vartheta$ maps $L_s \cap S$ to the $\mu$-axis and $L_t \cap S$ to the $\eta$-axis. Let 
\begin{equation*}
D_x = \xi \vartheta^ {-1}((-\tilde{s}_1, \tilde{s}_1) \times (-\tilde{t}_1, \tilde{t}_1))
\end{equation*}  
where $\tilde{s}_1, \tilde{t}_1 >0$ are so small that $\overline{\vartheta^ {-1}((-\tilde{s}_1, \tilde{s}_1) \times (-\tilde{t}_1, \tilde{t}_1))} \subset S$.
\begin{defi}
We call the data $\{C^\epsilon, D_x\}$  a \emph{regular neighborhood} of $C$.
\end{defi}

This means that a regular neighborhood of $C$ consists of an immersed annulus $C^\epsilon= \xi(S^1 \times(-\epsilon,\epsilon))$, and also a parametrization of a neighborhood of each double point of $C$ so that in this parametrization we have that $C$ coincides with the coordinate axes of $\R^2$. See Figure \ref{figure:regnbhd}.

\begin{figure}[h]
\labellist
\small\hair 2pt
\pinlabel $t$ [Br] at 8 253
\pinlabel $\epsilon$ [Br] at  8 233
\pinlabel $-\epsilon$ [Br] at 8 188
\pinlabel $S^1\times\R$ [Br] at 115 175
\pinlabel $\xi^{-1}(D_x)$ [Br] at 81 265
\pinlabel $1$ [Br] at 94 197
\pinlabel $s$ [Br] at 110 201
\pinlabel $\xi$ [Br] at 144 222
\pinlabel $C$ [Br] at 262 253
\pinlabel $C^\epsilon$ [Br] at 313 185
\pinlabel $D_x$ [Br] at 319 44
\pinlabel $\xi\circ\vartheta^{-1}$ [Br] at 267 69
\pinlabel $\mu$ [Br] at 220 25
\pinlabel $L_s$ [Br] at 215 7
\pinlabel $L_t$ [Br] at 206 66
\pinlabel $\eta$ [Br] at 167 80
\pinlabel $\vartheta$ [Br] at 125 70
\pinlabel $L_s$ [Br] at 90 67
\pinlabel $L_t$ [Br] at 87 114
\endlabellist
\centering
\includegraphics{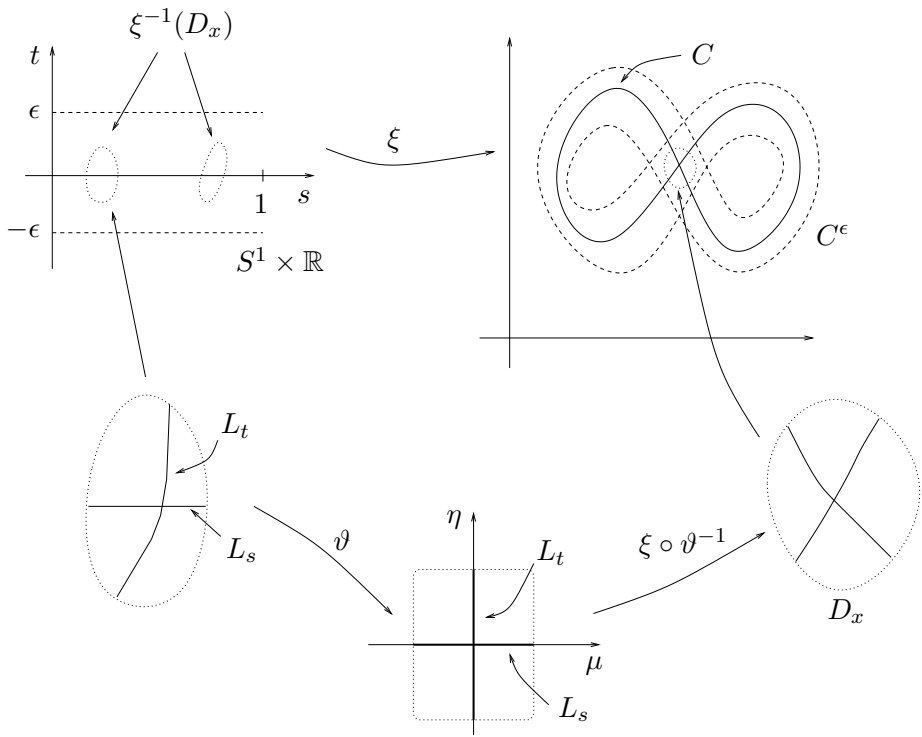}
\caption{An example of a regular neighborhood.}
\label{figure:regnbhd}
\end{figure}	  

Now let $C' \subset C^\epsilon$ be an immersion such that there exists a disk-area-preserving isotopy taking $C$ to $C'$. Let $Q_r$ be the open square with sides of length $2r$ centered at $0$, and $Q_{r,p}$ the open square with sides of length $2r$ centered at $p$. Let $\delta>0 $ be so small that for every double point $x \in C$ we have that $Q_{\delta,x}$ is contained in the parametrization of $D_x$. Further, for each double point $x \in C$, let $x'$ be the corresponding double point of $C'$, and let $L_s', L_t'\subset C'$ be the arcs corresponding to $L_s$ and $L_t$, respectively, in $C$. Assume that $C' \cap D_x \subset L_s' \cup L_t'$ and that $x' \in Q_{\delta,x}$ in the parametrization of $D_x$. Also assume that $L_s' \cap D_x $ and $L_t' \cap D_x $, respectively, are graphs of functions $g _\mu$ and $g _\eta$ over the $\mu$-and $\eta$-axis in the parametrization of $D_x$, satisfying $|g_\mu|,|g_\eta|, |\frac{\de g_\mu}{\de \mu}|, |\frac{\de g_\eta}{\de \eta}|< \delta$. If this holds for all double points of $C$, and if $C'$ is a graph of a function $g:S^1 \rightarrow (-\delta,\delta)$ in the parametrization of $C^\epsilon$ satisfying $|\frac{\de g}{\de s}| <\delta$, we say that $C'$ is \emph{$\delta$-close to $C$} in $\{C^\epsilon, D_x\}$.

The following result shows that if $C'$ is sufficiently close to $C$ in the above sense, then there is an area-preserving isotopy taking $C'$ to $C$.

\begin{lma}\label{delta}
Let $C$ be an immersion of $S^1$ in $\R^2$ and let $\{C^\epsilon, D_x\}$ be a regular neighborhood of $C$. Then there exists a $\delta>0$ such that for every immersion $C'$ which is $\delta$-close to $C$ in $\{C^\epsilon, D_x\}$ there is an area-preserving isotopy taking $C'$ to $C$. 
\end{lma}

\begin{proof}
Let $1>\sigma >0$ be so small so that in each parametrized disk $D_x$ we can find
a square $Q_{\sigma,x}=Q_\sigma$, where $x$ corresponds to $(0,0)$ in the
parametrization. Let $\delta>0$ be sufficiently small so that
$\sigma>\delta^{1/2}$  and let $\psi\colon\mathbb R \to \mathbb
R$ be a smooth cut-off function satisfying
\begin{equation*}
\psi(y)=
\begin{cases}
1	& \text{for } y\in(-\delta,\delta) \\
0	& \text{for } y \notin (-\sigma,\sigma)
\end{cases}
\end{equation*}
with 
\begin{align*}
|\psi|\leq 1, \qquad |\frac{d\psi}{d y}|<
\frac{a}{\delta^{1/2}-\delta}, \qquad
|\frac{d ^2\psi}{d y^2}|< \frac{b}{(\delta^{1/2}-\delta)^2}
\end{align*}
 for some constants $a$, $b$, i.e. 
\begin{equation*}
\frac{d\psi}{d y} = O(\delta^{-1/2}), \qquad \frac{d ^2\psi}{d
y^2}=O(\delta^{-1})
\end{equation*} 
as $\delta \to 0$.

Now let $C'$ be an immersion which is $\delta$-close to $C$ in $\{C^\epsilon,
D_x\}$, and let $x \in C$ be a double point. We start with showing that if
$\delta$ is sufficiently small then there is a neighborhood $U$ of $x$ and an
area-preserving isotopy $\phi_\tau$, $0\le \tau \le 1$, with support in $D_x$ so
that $\phi_1(C')\cap U$ coincides with $C \cap U$ and so that $\phi_1(C')$ is a
graph over $S^1$ in $C^\epsilon$. By finding one such isotopy for each double
point of $C$ and then use Lemma \ref{lma5} we get an area-preserving isotopy
taking $C'$ completely to $C$.

So given a double point $x\in C$, first consider the arc $L_s'\subset C'$,
defined as above. Since $C'$ is $\delta$-close to $C$ in $\{C^\epsilon, D_x\}$ 
we have that $L_s'$ coincides with the graph of a function $g_\mu \colon
(-\sigma,\sigma) \to (-\delta, \delta)$ in $Q_\sigma$. Let
$\delta$ be so small so that we can find an exact function $\tilde{g} \colon
\mathbb
R \to (-\delta, \delta)$ with support in $(-\delta^{1/2}, \delta^{1/2})$
whose graph coincides with $L_s'$ in $Q_\delta$ and which satisfies $|\frac{d
\tilde{g}}{d \mu}|=O(\delta^{1/2})$. Let $G(\mu)= \int_{-\sigma}^\mu
{\tilde{g}(\mu')
d \mu'}$, and consider the Hamiltonian $H(\mu, \eta)= -G(\mu)\psi(\eta)$ with
corresponding vector field 
\begin{equation*}
 X_H=G(\mu)\frac{d \psi}{d
\eta}(\eta)\frac{\partial}{\partial \mu} - \tilde{g}(\mu)\psi(\eta)
\frac{\partial}{\partial \eta}.
\end{equation*}
Then the Hamiltonian isotopy $\Phi_{X_H}^\tau = (\chi _\tau^1, \chi
_\tau^2)= \chi _\tau $, $0\le \tau \le 1$,
takes $L_s'$ to the $\mu$-axis in $Q_\delta$, and has support in $Q_\sigma$.

Next we want to take $\chi_1(L_t')$ to the $\eta$-axis in a neighborhood of $0$
in a way so that the image of  $\chi_1(L_s')$ still coincides with the
$\mu$-axis here. But first, to make sure that $\chi_1(C')$ is still a graph over
$S^1$ in the parametrization of $C^\epsilon$ we find an estimate for the
derivative $d\chi_1$  of $\chi_1$. 
Divide $[0,1]$ into $N$ intervals of length $1/N$. By Taylor expansion we have,
for $\tau \leq 1/N$, that
\begin{align*}
\frac{\partial \chi_\tau^1}{\partial \mu} &= 
\frac{\partial \chi_0^1}{\partial \mu} + \tau \frac{d}{d \tau}
\frac{\partial \chi_0^1}{\partial \mu} +O(\tau^ 2)= 
 1 + \tau \frac{\partial}{\partial \mu}(G(\mu)\frac{d \psi}{d \eta}(\eta)) +
O(\tau^ 2) = \\
&= 1 + \tau \tilde{g}(\mu)\frac{d \psi}{d \eta}(\eta) + O(1/{N^2})	
\end{align*}
 and
\begin{align*}
\frac{\partial \chi_{1/N+\tau}^1}{\partial \mu} &= 
\frac{\partial \chi_{1/N}^1}{\partial \mu} + \tau \frac{d}{d \tau}
\frac{\partial \chi_{1/N}^1}{\partial \mu} +O(\tau^ 2)= \\
&= (1 + 1/N \tilde{g}(\mu)\frac{d \psi}{d \eta}(\eta) + O(1/{N^2}))+
\tau\tilde{g}(\mu)\frac{d \psi}{d \eta}(\eta) + O(1/{N^2}).
\end{align*}
If we continue like this we get
\begin{align*}
\frac{\partial \chi_{1}^1}{\partial \mu} &= 
1 + \sum_{n=0}^ {N-1} {1/N \tilde{g}(\mu(n/N))\frac{d \psi}{d
\eta}(\eta(n/N))} + NO(1/{N^2}) \\
&= 1+ O(\delta^{1/2})+ O(1/N)
\end{align*}
since $|\tilde{g}|<\delta$, $|\frac{d \psi}{d \eta}|=O(\delta^{-1/2})$. Hence
for $N$ big
enough, depending on $C'$, we get $\frac{\partial \chi_{1}^1}{\partial \mu}= 1 +
O(\delta^{1/2})$, where the $O(\delta^{1/2})$-term depends on $C, C^\epsilon$
and $D_x$.
Similarly we have
\begin{equation*}
\frac{\partial \chi_{1}^1}{\partial \eta}= 0 +  \sum_{n=0}^ {N-1} {1/N
G(\mu(n/N))\frac{d^2 \psi}{d \eta^2}(\eta(n/N))} + O(1/{N})= O(\delta^{1/2})
\end{equation*}
since $|G|<\delta^{1/2}\delta$, $|\frac{d^2 \psi}{d \eta^2}|=O(\delta^{-1})$,
and
\begin{align*}
&\frac{\partial \chi_{1}^2}{\partial \mu}= 0 -  \sum_{n=0}^{N-1}{1/N
\frac{d\tilde{g}}{d \mu}(\mu(n/N))\psi(\eta(n/N))} + O(1/{N}) =
O(\delta^{1/2})\\
&\frac{\partial \chi_{1}^2}{\partial \eta}= 1 -  \sum_{n=0}^ {N-1} {1/N
\tilde{g}(\mu(n/N))\frac{d \psi}{d \eta}(\eta(n/N))} + O(1/{N})
=1+O(\delta^{1/2}).
\end{align*}
Thus we get that 
\begin{equation}
d\chi_1= E + O(\delta^{1/2})
\label{eq:diff}
\end{equation}
where $E$ is the $2 \times 2$ unit matrix and $O(\delta^{1/2})$ denotes a $2
\times 2$
matrix with entries of size $O(\delta^{1/2})$.

Now let $\vartheta=(\vartheta^1, \vartheta^2)\colon D_x \to D_x$
be a
change of coordinates from $(\mu,\eta)$ to $(s,t)\subset C^\epsilon$. In
$(s,t)$-coordinates by assumption we have that $L_s' \cap D_x = \{(s, g(s))\}$
for $s \in (\sigma_1, \sigma_2)$, say, and $g$ satisfies $|g|, |\frac{d g}{d
s}|< \delta$. By \eqref{eq:diff} we have 

\begin{equation*}
\frac{d}{d s}\vartheta^1(\chi_1 \vartheta^{-1}(s,g(s))) = 1 + O(\delta^{1/2})
\end{equation*}
for all $s\in (\sigma_1, \sigma_2)$, so $\chi_1(L_s')$ is a graph of a function
$\alpha\colon S^1 \to \mathbb R$ in the parametrization of
$C^\epsilon$ if we let $\delta$ be small enough.
 Furthermore, for the slope of $\alpha$ we get that 
 \begin{align*}
\left|\frac{d \alpha}{d s}\right| =\left| \frac{\frac{d}{d
s}\vartheta ^2(\chi_1 \vartheta^{-1}(s,g(s)))}{\frac{d}{d s}\vartheta
^1(\chi_1 \vartheta^{-1}(s,g(s)))}\right|=
\frac{O(\delta^{1/2})}{1+O(\delta^{1/2})}=O(\delta^{1/2}). 
\end{align*}

Similar calculations show that $\chi_1(L_t')$ is a subset of both a graph over
$S^1$ in the parametrization of $C^\epsilon$ and a graph over the $\eta$-axis
in the parametrization of $D_x$ for $\delta$ sufficiently small. Moreover, the
slope of these graphs are of order $\delta^{1/2}$. 

Now we find an isotopy $\tilde{\chi}_\tau$, $0\le \tau \le 1$, taking
${\chi}_1(L_t')$ to $L_t$ in a neighborhood of $x$, and so that 
$\tilde{\chi}_1(\chi_1(L_s'))$ still coincides with $L_s$ here. Since by
assumption we had $x' \in Q_\delta \subset D_x$, where $x'\in C'$ is the double
point corresponding to $x$, we have $\chi_1(x') \in
(-\delta,\delta)\times\{0\}$. Hence we can find a $0<\delta'<\delta$ so that
$\chi_1(L_t')$ coincides with the graph of an exact function $f\colon \mathbb R
\to (-\delta,\delta)$ in $(-\delta,\delta) \times (-\delta',
\delta')$, that is,  $\chi_1(L_t')\cap ((-\delta,\delta) \times (-\delta',
\delta'))=\{(f(\eta),\eta)\}$. In addition we can choose $f$ so that 
$|\frac{d f}{d\eta}| =O(\delta^{1/2})$ for all $\eta \in \mathbb R$ and so that 
$f(\eta)=0$ for $|\eta|> \delta^{1/2}$. Let
$F(\eta)=\int_{-\sigma}^\eta{f(\eta') d
\eta'}$. Then the isotopy $\Phi_{X_H}^\tau =\tilde{\chi}_\tau$, $0\le \tau \le
1$, obtained from the Hamiltonian  $H(\mu,\eta)=\psi(\mu)F(\eta)$ takes
$\chi_1(L_t')$ to the $\eta$-axis in $(-\delta, \delta)
\times(-\delta',\delta')$, and we have that $\tilde{\chi}_1\chi_1(L_s')$ still
coincides with  the $\mu$-axis in a neighborhood of
$(0,0)=\tilde{\chi}_1\chi_1(x')$. 

As before we get that 
\begin{equation*}
d\tilde{\chi}_1 = E +
\begin{bmatrix}
\frac{d \psi}{d \mu} f & \psi \frac{d f}{d\eta} \\
\frac{d^2 \psi}{d \mu^2} F & \frac{d \psi}{d \mu} f 
\end{bmatrix}
+ O(1/N) = E+ O(\delta^{1/2})
\end{equation*} 
for $N$ large. 
So for $\tilde{\chi}_1\chi_1(L_s')$ in $C^\epsilon \cap D_x$ we have, with
$\chi_1(L_s')=\{(s, \alpha(s))\}$ here, that 
\begin{equation*}
\frac{d}{d s}\vartheta \tilde{\chi}_1 \vartheta ^{-1}(s, {\alpha}(s)) =
\begin{bmatrix}
1 & 0 \\
0 & \frac{d \alpha}{d s}  
\end{bmatrix}
+ O(\delta^{1/2}).
\end{equation*}
Hence $\tilde{\chi}_1 \chi_1(L_s')$ will be a subset of a graph over $S^1$ for
$\delta$ small enough, and similarly we get that $\tilde{\chi}_1 \chi_1(L_t')$
is a subset of a graph over $S^1$ in the parametrization of $C^\epsilon$ too. 

By doing the same thing at all double points of $C$ we get an area-preserving
isotopy taking $C'$ to $C$ in a neighborhood of every double point of $C$, and
so that the time $1$-image of $C'$ is still a graph over $S^1$ in $C^\epsilon$.
So
by Lemma \ref{lma5} there is an area-preserving isotopy taking $C'$ completely
to $C$. 
\end{proof}

\section{Proof of Theorem \ref{thm1}}\label{sec:proof}
Now if we combine Lemma \ref{sapi} with Lemma \ref{delta} we can prove our theorem:
\begin{proof}[Proof of Theorem \ref{thm1}]
By Lemma \ref{sapi} there is a semi-area-preserving isotopy $\phi_\tau$, $0\le \tau \le 1$, with respect to $C$ taking $C$ to $C'$. Let $C_\tau=\phi_\tau(C)$ for $\tau \in [0,1]$, and for each $\tau_0 \in [0,1]$ let $\{C_{\tau_0}^\epsilon, D^{\tau_0}_x\}$ be a regular neighborhood of $C_{\tau_0}$. By Lemma \ref{delta} we can find a $\delta_{\tau_0}>0$ so that for every $C_\tau$ which is $\delta_{\tau_0}$-close to $C_{\tau_0}$  there exists an area-preserving isotopy taking $C_\tau $ to $C_{\tau_0}$,  and by the continuity of $\phi_\tau$ there is a $\nu_{\tau_0}>0$ so that $C_\tau$ is $\delta_{\tau_0}$-close to $C_{\tau_0}$ for all $0\leq \tau-\tau_0 < \nu_{\tau_0}$. 

Let $\nu = \text{min}_{\tau_0 \in I} \{\nu_{\tau_0}\}$ and let 
\begin{equation*}
0=  \tau_1 < \dotsc < \tau_n=1
\end{equation*}
be a partition of $[0,1]$ so that $\tau_{i+1}-\tau_i < \nu$ for $1\leq i <n$. Then by Lemma \ref{delta} there is an area-preserving isotopy taking $C_{\tau_{i+1}}$ to $C_{\tau_i}$ for $i=1,\dotsc,n-1$. Composing the inverses of these isotopies we thus get an area-preserving isotopy taking $C$ to $C'$.
\end{proof}

\bibliographystyle{halpha}
\bibliography{kand} 

\end{document}